\documentclass[final,1p,authoryear]{elsarticle}
\usepackage{amsmath}
\usepackage{amssymb}
\usepackage{amsthm}
\usepackage{amsfonts}
\usepackage{geometry}
\usepackage{hyperref}
\hypersetup{
    colorlinks=true,
    linkcolor=blue,
    filecolor=magenta,      
    urlcolor=cyan,
    }
\usepackage[utf8]{inputenc}
\usepackage[vietnamese,english]{babel}
\usepackage{comment}

\makeatletter
\def\ps@pprintTitle{%
   \let\@oddhead\@empty
   \let\@evenhead\@empty
   \let\@oddfoot\@empty
   \let\@evenfoot\@oddfoot
}
\makeatother

\theoremstyle{plain}
\newtheorem{theorem}{Theorem}[section]

\newtheorem{proposition}[theorem]{Proposition}
\newtheorem{lemma}[theorem]{Lemma}
\theoremstyle{definition}

\newtheorem{remark}[theorem]{Remark}

\newcommand{\F}{\mathbb{F}}
\newcommand{\Q}{\mathbb{Q}}
\newcommand{\C}{\mathbb{C}}
\newcommand{\N}{\mathbb{N}}
\newcommand{\Z}{\mathbb{Z}}

\newcommand{\R}{\mathbb{R}}

\newcommand{\st}{{|}}
\newcommand{\frm}{\mathfrak{m}}

\newcommand{\ord}{{\rm ord}}

\newcommand{\id}{{\rm id}}

\newcommand{\calC}{{\cal{C}}}

\newcommand{\vp}{{\varphi}}

\begin{document}

\begin{frontmatter}

\title{A skew Newton-Puiseux Theorem}

\author[OPENU]{Elad Paran}
\ead{paran@openu.ac.il}

\author[TDTU]{Thieu N. Vo\corref{mycorrespondingauthor}}
\cortext[mycorrespondingauthor]{Corresponding author}
\ead{vongocthieu@tdtu.edu.vn}

\address[OPENU]{The Open University of Israel}

\address[TDTU]{Fractional Calculus, Optimization and Algebra Research Group, Faculty of Mathematics and Statistics, Ton Duc Thang University, Ho Chi Minh City, Vietnam}

\begin{abstract}
We prove a skew generalization of the Newton-Puiseux theorem for the field $F = \bigcup_{n=1}^\infty \C((x^\frac{1}{n}))$ of Puiseux series: For any positive real number $\alpha$, we consider the $\C$-automorphism $\sigma$ of $F$ given by $x \mapsto \alpha x$, and prove that every non-constant polynomial in the skew polynomial ring $F[t,\sigma]$ factors into a product of linear terms. This generalizes the classical theorem where $\sigma = \id$, and gives the first concrete example of a field of characteristic $0$ that is algebraically closed with respect to a non-trivial automorphism -- a notion studied in works of Aryapoor and of Smith. Our result also resolves an open question of Aryapoor concerning such fields. A key ingredient in the proof is a new variant of Hensel's lemma.

\end{abstract}

\end{frontmatter}

\section{Introduction}

Let $F$ be a field and let $\sigma$ be an endomorphism of $F$. We shall say that $F$ is {\bf algebraically closed with respect to $\sigma$}, if every non-constant polynomial in the skew polynomial ring $F[t,\sigma]$ factors into a product of linear terms (equivalently, if every such polynomial has a $\sigma$-zero, in the sense introduced in \cite{LL88}). This notion was studied in \cite{Sm77} and in the recent paper \cite{Ar23}, with different terminologies: Aryapoor calls such a field $F$ a $1$-algebraically closed $\sigma$-field, see \cite[\S4.1]{Ar23}, while Smith refers to the ring $F[t,\sigma]$ itself as algebraically closed, see \cite[p.~332]{Sm77}. The first non-trivial example (where $\sigma \neq \id$) for such a field appears in the classic paper \cite{Ore33b}, who proved that the algebraic closure of a finite field is algebraically closed with respect to the Frobenius automorphism.

An interesting result of \cite{Ar23} is that every field $F$ equipped with an automorphism $\sigma$ can be embedded in a field which is algebraically closed with respect to an extension of $\sigma$. This theorem is proven via transfinite induction, and does not yield any concrete examples. In fact, it seems that the only explicit example in the literature is the mentioned one of Ore, and in particular no concrete examples in characteristic $0$ are given in any of the works mentioned above.

In the present work, we consider this notion in the context of the classical Newton-Puiseux theorem, which states that the field of Puiseux series $F = \bigcup_{n=1}^\infty \C((x^\frac{1}{n}))$ is algebraically closed \cite[Corollary 1.2, p.~112]{Coh91}. We prove the following generalization of this theorem: \begin{theorem}\label{main:intro} Let $\alpha$ be a positive real number. Let $\sigma$ be the $\C$-automorphism of $F = \bigcup_{n=1}^\infty \C((x^\frac{1}{n}))$ determined by $\sigma(x^\frac{1}{n}) = (\alpha x)^\frac{1}{n}$ for all $n \geq 1$. Then $F$ is algebraically closed with respect to $\sigma$.\end{theorem}

The case where $\alpha = 1$ (equivalently, $\sigma = \id$) of this result is the classical Newton-Puiseux theorem. The case where $\alpha \neq 1$ provides the first concrete example of a characteristic $0$ field which is algebraically closed with respect to a non-trivial automorphism. Theorem \ref{main:intro} also resolves an open question raised in \cite{Ar23}, see Remark \ref{open} and Remark \ref{final} below.

One ingredient in the proof of Theorem \ref{main:intro} is a non-commutative variant of Hensel's lemma, Proposition \ref{Hensel}, which may be of independent interest. Hensel's lemma has been previously studied in the context of skew polynomial rings in \cite{Ar09},  however the version we present here involves a more subtle condition on the residue polynomials than the usual one. We provide examples in \S\ref{skew_hensel} 
that demonstrate why our condition is the ``right" condition for skew polynomial rings. To the best of our knowledge, this variant does not appear elsewhere in the literature. The utility of Proposition \ref{Hensel} is further demonstrated via its usage in the proof of Theorem \ref{main:intro} in \S\ref{sec:newton}.

Our general proof strategy of Theorem \ref{main:intro} is similar to that of the classical theorem: Starting with a general polynomial, we apply a serious of reduction steps in order to bring it to a form for which Hensel's lemma applies. However, when making these reduction steps some complications arise that do not have a counter-part in the classical proof via Hensel's lemma (see for example in \cite[Lecture 12]{Ab90}): The main difficulty is that linear changes of variable do not induce automorphisms of the ring $F[t,\sigma]$, but rather induce a change of structure which involves $\sigma$-derivations (in particular, see Lemma \ref{shift} and Lemma \ref{lem:isom}). Thus in order to prove the theorem for $F[t,\sigma]$ we are forced to factor certain polynomials in a more general ring, of the form $F[t,\sigma, \delta]$, where $\delta$ is an inner $\sigma$-derivation. Therefore, before proving the main theorem, we establish a series of lemmas and propositions that track the properties of our polynomial through several steps of structure change.

This paper is organized as follows: In \S\ref{prelim} we quickly review some basic material concerning general skew polynomial rings and substitution in such polynomials. In \S\ref{skew_hensel} we establish our variant of Hensel's lemma. In \S\ref{sec:newton} we prove Theorem \ref{main:intro} via the approach discussed above.

\section{Preliminaries}\label{prelim}

Throughout this work, all rings are assumed to be associative and unital.

Let $A$ be a (possibly noncommutative) ring, let $\sigma$ be an endomorphism of $A$ and let $\delta$ be a $\sigma$-derivation of $A$. That is, $\delta \colon A \to A$ is an additive map satisfying the ``Leibniz rule" $\delta(ab) = a^\sigma \delta(b)+\delta(a)b$ for all $a,b \in A$. Then we may form the classical skew polynomial ring $R = A[t,\sigma,\delta]$, where multiplication is determined by the rule $ta = a^\sigma t + \delta(a)$ for all $a \in A$. 

For a standard introduction to skew polynomial rings, see  \cite[\S2]{Coh95}, or \cite[\S2]{GW04} (see also \cite{Ore33}, the seminal paper on the subject). 

For any monic polynomial $p \in R$ and for any $f \in R$ one can perform left-hand division with remainder and find uniquely determined polynomials $q,r \in R$ such that $f = qp+r$ and $\deg(r) < \deg(p)$. In the case where $p = x-a$ is a linear polynomial, the remainder $r$ is a constant in $A$ whose value is the {\bf $(\sigma,\delta)$-substitution} of $a$ in $f$, denoted by $f(a)$. In particular, a polynomial $f$ is right-hand divisible by $x-a$ if and only if $f(a) = 0$, in which case we say that $a$ is a {\bf $(\sigma,\delta)$-zero} of $f$, or just a {\bf $\sigma$-zero} of $f$ in case $\delta = 0$. If $\delta = 0$ and $f$ is written in the form $f = \sum f_i t^i$, then $f(a)$ is given by:

$$f(a) = f_0 + f_1 a +f_2a^\sigma a+ f_3 a^{\sigma^2}a^\sigma a +\ldots.$$
For further details on substitution in skew polynomial rings, see \cite{LL88} (we note that in \cite{LL88} the ring $A$ is assumed to be a division ring, but the above claims hold for a general ring).

The following proposition is given in \cite[Proposition 2.4, p.~37]{GW04} (with $A,\sigma,t$ here corresponding to $R,\alpha,x$ there, respectively). We formulate it here, as it will be used at several points in the sequel:

\begin{proposition}\label{goodearl}
    Let $S = A[t,\sigma,\delta]$ be a skew polynomial ring. Let $T$ be a ring, let $\phi \colon A \to T$ be a ring homomorphism, and let $y \in T$ be an element satisfying $y\phi(a) = \phi(a^\sigma)y+\phi(\delta(a))$ for all $a \in A$. Then there exists a unique homomorphism $\psi \colon S \to T$ satisfying $\psi|_A = \phi$ and $\psi(t) = y$. 
\end{proposition}

We note that given an endomorphism $\sigma$ of a ring $R$ we shall sometimes use the notation $\sigma(a)$ and sometimes the notation $a^\sigma$, to denote the image of an element $a\in R$ under $\sigma$, according to convenience of presentation.

Let $\frm$ be a (two-sided) ideal in $A$, satisfying $\sigma(\frm) \subseteq \frm$ and $\delta(\frm) \subseteq \frm$. Then $\sigma$ induces an endomorphism of the quotient ring $A/\frm$, given by $\bar{\sigma}(a+\frm) = \sigma(a)+\frm$, and $\delta$ induces a $\bar{\sigma}$-derivation $\bar{\delta}$ on $A/\frm$, given by $\bar{\delta}(a+\frm) = \delta(a)+\frm$. Thus we may form the skew polynomial ring $A/\frm[t,\bar{\sigma},\bar{\delta}]$. 

\begin{lemma}\label{reduction} 
The quotient map $a \mapsto \bar{a}$ from $A$ to $A/\frm$ naturally extends to an epimorphism $p \mapsto \bar{p}$ from $A[t,\sigma,\delta]$ onto $A/\frm[t,\bar{\sigma},\bar{\delta}]$, given by $\sum a_i t^i \mapsto \sum \bar{a}_i t^i$. \end{lemma} \begin{proof}
    
Apply Proposition \ref{goodearl} with $T =A/\frm[t,\bar{\sigma},\bar{\delta}], y = t$, and $\phi$ being the map $a \mapsto \bar{a}$ from $A$ to $\bar{A} \subseteq T$. The obtained homomorphism $\psi$ must then coincide with the surjective map $p \mapsto \bar{p}$ considered here, since they agree on $t$ and on $A$. \end{proof}

As mentioned in the introduction, given a field $K$ and an endomorphism $\sigma$ of $K$, we say that $K$ is algebraically closed with respect to $\sigma$ if every non-constant polynomial in the ring $K[t,\sigma]$ factors into a product of linear terms. Equivalently (by induction on the degree), if and only if for each non-constant $f \in K[t,\sigma]$ there exists a $\sigma$-zero in $K$. Of course, every algebraically closed field is algebraically closed with respect to the identity automorphism. However, the field of complex numbers is not algebraically closed with respect to complex conjugation, for example. Indeed, let us denote complex conjugation by $\rho \colon \C \to \C$, and consider the polynomial $f = t^2+i \in \C[t,\rho]$. For any $a \in \C$ we have $a^\rho a \in \R$, hence $f(a) \neq 0$. However, one can show that $\C$ is ``close" to being algebraically closed with respect to $\rho$, in the following sense: Every non-constant polynomial in $\C[t,\rho]$ factors into a product of linear or quadratic terms, see \cite[Corollary 6]{Pum15} or \cite[Theorem 3.1]{BGS15}. Note also that in skew polynomial rings, a non-zero polynomial may have more zeros that its degree. For example, in the above example, every complex number with norm $1$ is a $\rho$-zero of $t^2-1 \in \C[t,\rho]$. 

\begin{remark}\label{open}
It is shown in \cite[Theorem 3]{Sm77} that if a field $K$ is algebraically closed with respect to a non-trivial endomorphism $\sigma$, then $\sigma$ must be of infinite order. Smith only gives one concrete example of a field which is algebraically closed with respect to a non-trivial endomorphism: The algebraic closure $\bar{\F}_p$ of a finite field, with respect to the Frobenius automorphism $\sigma$, see \cite[p.~345]{Sm77} (Smith credits this result to Ore -- it is an immediate consequence of \cite[Theorem 3]{Ore33b}). This example is also discussed in \cite[\S4.3]{Ar23}\footnote{Aryapoor seems unaware of Smith's paper. In particular, Proposition 4.1 in \cite{Ar23} is a special case of Smith's Theorem 3 in \cite{Sm77}.}, where Aryapoor shows that every polynomial  in $\bar{\F}_p[t,\sigma]$ of degree larger than $1$, with a non-zero constant term, has at least $p+1$ distinct $\sigma$-zeros. In light of this property, Aryapoor raises in \cite[\S4.3]{Ar23} the following theoretical question: Does there exist a field $K$, algebraically closed with respect to an automorphism $\sigma \neq \id$, and a polynomial $f \in K[t,\sigma]$ of degree larger than $1$ with a non-zero constant term, that has only one $\sigma$-zero? As a consequence of our main result, we resolve this question to the affirmative, see Remark \ref{final}.
(In Aryapoor's terminology, we have proven the existence of a nontrivial $1$-algebraically closed $\sigma$-field which is not $2$-algebraically closed.)
\end{remark}

\section{A skew Hensel's lemma}\label{skew_hensel}

For this section, let $A$ be a (possibly non-commutative) domain, complete and separated with respect to a principal two-sided ideal $\frm = xA = Ax$, for some $x \in A$. We further assume that $K = A/\frm$ is a (commutative) field. 

Let $\sigma$ be an endomorphism of $A$ and let $\delta$ be a $\sigma$-derivation of $A$, such that $\sigma(\frm) \subseteq \frm$ and $\delta(\frm) \subseteq \frm$. Let $\bar{\sigma}$ and $\bar{\delta}$ be the induced endomorphism and $\bar{\sigma}$-derivation of $K$, as in \S\ref{prelim}. We shall further assume here that $\bar{\sigma} = \id_K$ and $\bar{\delta} = 0$. By Lemma \ref{reduction}, we have a reduction epimorphism $\sum a_i t^i \mapsto \overline{\sum a_it^i} = \sum \overline{a}_it^i$ which maps $A[t,\sigma,\delta]$ onto the commutative ring $K[t,\id,0] = K[t]$.

\begin{remark} The above setup applies to classical commutative settings (with $\sigma = \id, \delta = 0$), for example for $A = \Z_p$, the ring of $p$-adic integers (with $\frm = pA$), or for $A = K[[x]]$, the ring of formal power series over a field $K$ (with $\frm = \langle x \rangle$). However, this setup also applies to various noncommutative settings: For example, we may take $A = K[[x,\tau]]$, the ring of skew formal power series over a field $K$, where $\tau$ is an automorphism of $K$ and $\frm = xA = Ax$, $\sigma = \id_A$,$\delta= 0$. Or, let $A = K[[x]]$ be the usual ring of formal power series over $K$, let $\frm = xA$, and let $\sigma$ be a $K$-automorphism of $A$ determined by $\sigma(x) = \alpha x$ for some fixed $\alpha \in K^\times$, and with $\delta = 0$.


\end{remark}

Let us denote $R = A[t,\sigma,\delta]$. Since $A$ is a domain, so is $R$. For any $a \in A$, we denote by $\ord(a)$ the $x$-adic order of $a$, given by $\ord(a) = \max\{k | a \in x^k A\}$ for $a \neq 0$, and $\ord(0) = \infty$. We extend this order map to $R$ by $\ord(\sum a_it^i) = \min_i \ord(a_i)$.

\begin{lemma}\label{two-sided} Let $n \in \N$. The set $\{p \in A[t,\sigma,\delta] \st \ord(p) \geq n\}$ is a two-sided principal ideal in $A[t,\sigma,\delta]$, generated by $x^n$. \end{lemma}

\begin{proof} 
By definition, the given set is: 
\begin{align*}
    &\left\{\sum a_i t^i \in A[t,\sigma,\delta] \,\st\, \ord(a_i) \geq n \text{ for all } i\geq 0\right\}\\
    &\qquad \qquad =\left\{\sum x^nb_i t^i \in A[t,\sigma,\delta] \,\st\, b_i \in A \text{ for all } i\geq 0\right\} = x^nR.
\end{align*}
%
%
It is left to prove that $x^nR= Rx^n$. By our assumptions, we have $x^\sigma,\delta(x) \in xA = Ax$, hence $tx = x^\sigma t+\delta(x) \in (xA)R = xR$. It follows by induction that $t^ix \in xR$ for all $i \geq 0$,  hence for all $a \in A$ we have $at^ix \in axR \subseteq AxR = xAR = xR$. Since $R$ is generated additively by monomials of the form $at^i$, $a \in A$, we get that $Rx \subseteq xR$. Similarly one proves that $xR \subseteq Rx$. Thus $Rx = xR$, and by induction it follows that $Rx^n = x^n R$ for all $n \geq 1$. \end{proof}

By Lemma~\ref{two-sided}, for any $p \in R$ there exists an element $p^\phi \in R$ such that $p^\phi x = xp$. Since $R$ is a domain, $p^\phi$ is uniquely determined by $p$, and the map $\phi:p \mapsto p^\phi$ is clearly an automorphism of $R$. Note that $x^\phi = x$.

We now present the following variant of Hensel's lemma:

\begin{proposition} \label{Hensel}
Let $f \in A[t,\sigma,\delta]$ be a monic polynomial in $t$ of degree $d$. Let $g,h \in A[t,\sigma,\delta]$ be monic polynomials of respective degrees $m,d-m$ with $\bar{f} = \bar{g}\bar{h}$, such that $\bar{g},\overline{h^{\phi^{n}}}$ are relatively prime in $K[t]$ for all $n \in \N$. 
Then there exist monic polynomials $\hat{g},\hat{h} \in A[t,\sigma,\delta]$ of respective degrees $m,d-m$ with $\overline{\hat{g}} = \bar{g}, \overline{\hat{h}} = \bar{h}$, such that $f = \hat{g}\hat{h}$. 
\end{proposition}

Let us emphasize that the conditions of Proposition \ref{Hensel} do not require the residue polynomials $\bar{g},\overline{h}$ to be coprime in $K[t]$. Instead we require $\bar{g}$ to be coprime to each of the polynomials in the sequence $\{\overline{h^{\phi}}$, $\overline{h^{\phi^2}}$, $\overline{h^{\phi^3}},\ldots\}$. Note that if $A$ is a complete discrete valuation ring (DVR) and $R$ is the usual commutative polynomial ring over $A$ (that is, $\sigma = \id$, and $\delta =0$) then Proposition \ref{Hensel} coincides with the usual form of Hensel's lemma for $R = A[t]$. Indeed, in this case the automorphism $\phi$ is the identity and so the polynomials $\overline{h^{\phi}}$, $\overline{h^{\phi^2}}$, $\overline{h^{\phi^3}},\ldots$ all coincide with $\bar{h}$. Thus the condition of Proposition \ref{Hensel} in the commutative case is the usual well-known condition.

Before proving Proposition \ref{Hensel}, let us consider some examples that illustrate why the above condition is the ``correct" condition for skew polynomial rings.

\subsection{Example 1}

Let $A = \C[[x]]$ be the usual formal power series ring, let $\alpha$ be some positive real number, and let $\sigma$ be the $\C$-automorphism of $A$ given by $\sigma(x) = \alpha x$. Let $a,b \in A$ with $\ord(a), \ord(b) > 0$, that is, $a=\sum_{k \geq 1}a_kx^k,  b=\sum_{k \geq 1}b_kx^k  \in x\C[[x]]$. Consider the polynomial $f = t^2-(2+a)t + 1+b$ in the ring $A[t,\sigma]$ (in this example $\delta = 0$. That is, no derivation is involved). Suppose further that $a_1 \neq b_1$.

\begin{proposition} The polynomial $f = t^2-(2+a)t+1+b$ factors non-trivially in $A[t,\sigma]$ if and only if $\sigma \neq \id$ (equivalently, if and only if $\alpha \neq 1$). \end{proposition}

Let us first prove the claim directly, without employing Hensel's lemma, by an explicit analysis:

\begin{proof}

By division with remainder, the polynomial $f$ factors non-trivially in $A[t,\sigma]$ if and only if there exists a $\sigma$-zero $g\in A$ for $f$. That is, an element $g = \sum_{i \geq 0} g_i x^i$ satisfying the equality 

$$\left(\sum g_i \alpha ^i x^i\right) \left(\sum g_i x^i\right) -\left(2+\sum a_i x^i \right) \left(\sum g_i x^i \right)+1+\sum b_i x^i = 0$$
in $\C[[x]]$. By comparing coefficients, we must have  $g_0^2-2g_0+1 = (g_0-1)^2 = 0$, hence $g_0 = 1$. Suppose that we have managed to find $g_0,\ldots,g_n$ such that coefficients of $1,x,\ldots,x^{n}$ in the left-hand side of the presented equality are $0$. Then the coefficient of $x^{n+1}$ in the left-hand side is of the form $$g_{n+1}(\alpha^{n+1}+1) - 2g_{n+1} + (*) = g_{n+1}(\alpha^{n+1}-1)+(*),$$
where $(*)$ is an expression that depends only on $g_0,\ldots,g_{n},a_1,\ldots,a_{n+1}$ and $b_1,\ldots,b_{n+1}.$ From this we see that if $\alpha \neq 1$ we may inductively choose $g_1,g_2,\ldots$ such that the all of the presented coefficients will vanish. 

Conversely, if $\alpha = 1$, then by looking at the coefficient of $x$ we must have $$g_1(1+\alpha)-(2g_1+a_1)+b_1 = b_1-a_1 = 0,$$ 
contradicting our assumption $a_1 \neq b_1$. Thus for $\alpha = 1$ the polynomial $f$ does not factor.
\end{proof}

Let us view the preceding claim with Hensel's lemma in mind: Since $a,b \in x\C[[x]]$, the residue polynomial $\bar{f}$ is $t^2-2t+1 = (t-1)^2$, which of course does not factor into a product of coprime polynomials in $\C[t]$. Thus the condition of the classical form of Hensel's lemma fails, and indeed in the commutative case $\alpha = 1$ the polynomial $f$ does not factor. However,  if $\alpha \neq 1$, we can easily show via Proposition \ref{Hensel} that $f$ factors: For $g = h = t-1 \in A[t,\sigma]$ we have $\bar{f} = \bar{g}\bar{h}$, and we have $(\alpha^{-1}t)x =\alpha^{-1}(\alpha x t) = x t$, hence 
$\phi(t) = \alpha^{-1}t$.  By induction we get: $$\phi^n(h) = \phi^n(t)-1 = \alpha^{-n}t-1$$
for all $n \in \N$. If $\alpha = 1$ we simply get $\phi^n(h) = t-1$ for all $n$, but if $\alpha \neq 1$ then the polynomial $$\overline{\phi^n(h)} = \alpha^{-n}t-1$$
is coprime to $\overline{\phi(g)} = t-1$ in $\C[t]$ for all $n \in \N$, and so the condition of Proposition \ref{Hensel} is met. Thus $f$ factors in $A[t,\sigma]$.

\subsection{Example 2}

Let $A$ be the skew formal power series ring $\C[[x,\rho]]$, where $\rho$ denotes complex conjugation, and let $f = t^2+(1+x)$. In \cite[Example 4.2]{Ar09}, Aryapoor observes that $f$ does not factor in the polynomial ring $A[t]$, despite the fact that the residue polynomial $t^2+1$ factors into a product of coprime polynomials in $\C[t]$: For $g = t+i,h = t-i$, we have $\bar{f} = \bar{g}\cdot\bar{h}$. Thus here the ``usual" form of Hensel's lemma fails. This phenomena was also noted in \cite[\S5]{Vel10}, in a different context. 

Let us revisit this example in light of Proposition \ref{Hensel}: In this example the variable $t$ is central, hence $\phi(t)= t$. However, we have $xi = -ix$, hence  $\phi(i) = -i$, and so $\phi(h) = t+i$. Thus the polynomials $\overline{g} = t+i$ and $\overline{\phi^1(h)} = t+i = \bar{g}$ are {\bf not} coprime in the residue ring $\C[t]$, hence the failure of Hensel's lemma in factoring $f$ is, in fact, not unexpected: The condition of Proposition \ref{Hensel} is not met. 

\subsection{Example 3}

Let $K$ be a field equipped with a derivation, and let $A$ be the ring of Volterra operators over $K$: The elements of $A$ are formal power series of the form $\sum_{i=0}^\infty u_i x^i$ with $u_0,u_1,\ldots \in K$, and multiplication is determined by the rule:

$$xu = \sum_{i=0}^\infty (-1)^iu^{(i)}x^{1+i} = ux-u^{(1)}x^2+u^{(2)}x^3-u^{(3)}x^4+\ldots$$

for all $u \in K$, where $u^{(i)}$ denotes the $i$-th derivative of $u$. In \cite[Example 4.1]{Ar09}, Aryapoor shows that for polynomials in the ring $R = A[t]$ (where $t$ is a central variable), Hensel's lemma in its classical form (that is, with the usual condition on the residue polynomials) holds true. Let us understand this example in light of Proposition \ref{Hensel}:

In the ring $R$, we have $t^\phi = t$, since $t$ is central, and by the above product rule we get
$$u^\phi = u-u^{(1)}x+u^{(2)}x^2-u^{(3)}x^3+\ldots,$$
hence $\overline{u^\phi} = \bar{u}$, for all $u \in K$. For any $a \in xA = Ax$ we clearly have $a^\phi \in xA$, hence $\overline{a} = \overline{a^\phi} =0 $. Thus for a general element $a = \sum_{i=0}^\infty u_i x^i \in A$, we have 

$$\overline{a^\phi} = \overline{u_0^\phi} = \overline{u_0} = \overline{a}.$$
It follows that for any monomial $at^i$ with $a \in A$ we have $$\overline{(at^i)^\phi} = \overline{a^\phi(t^i)^\phi} = \overline{a^\phi}\overline{t^i} =\overline{a}\overline{t^i} = \overline{at^i},$$
hence $\overline{h^\phi} = \overline{h}$ for all $h \in R = A[t]$. 

The condition of Proposition \ref{Hensel} requires $\bar{g}$ to be coprime to each of the polynomials $\overline{h^{\phi^n}}$, but the latter all coincide with $\bar{h}$, by the above observations. Thus in this example the condition of Proposition \ref{Hensel} simply coincides with the usual condition of the classical Hensel's lemma. 

\subsection{Proof of Hensel's lemma}

The above examples demonstrate why the condition of our version of Hensel's lemma is the ``right" condition for skew polynomial rings.  In the next section, we give our main application of this result, as a key ingredient in the proof of Theorem~\ref{main:intro}.

Let us now prove Proposition \ref{Hensel}. 

\begin{proof} 
By Lemma \ref{two-sided}, $xR = Rx$ is a two-sided ideal in $R = A[t,\sigma,\delta]$. Since $A$ is separated and complete with respect to the ideal $xA$, clearly $R$ is separated and complete with respect to $xR$. For each $n \geq 1$, we build polynomials $p_{n-1}, q_{n-1} \in R$ with $\deg(p_i) < m,\deg(q_i) \leq d-m$ (throughout this proof, the ``degree" $\deg$ notation refers to the degree in $t$ of a polynomial in $R$) such that 
\begin{align*}
   &g_{n-1} = g+p_0 + p_1 x + \ldots + p_{n-1}x^{n-1}, \text{ and }\\
   &h_{n-1} = h+q_0 + q_1 x + \ldots + q_{n-1}x^{n-1}  
\end{align*}
satisfy $\ord(f-g_{n-1}h_{n-1}) \geq n$. If we manage this then we may take limits as $n \to \infty $ to obtain the necessary $\hat{g},\hat{h}$.

For $n = 1$ we just take $p_0 = q_0 = 0$, so that $g_{0} = g, h_0 = h$. Then $\bar{f} = \bar{g}\bar{h} = \overline{g_0}\overline{h_0} = \overline{g_0h_0}$, hence $\ord(f-g_{0}h_{0}) \geq 1$. Suppose we have already constructed the necessary $p_0,q_0,p_1,q_1,\ldots,p_{n-1}, q_{n-1}$, for a given $n$. Then $f-g_{n-1}h_{n-1}$ is divisible (from both left and right, by Lemma \ref{two-sided}) by $x^n$. Write $f-g_{n-1}h_{n-1} =  f_nx^n$ for a suitable $f_n \in R$. Note that $\deg(g_{n-1}) = m$ and $\deg(h_{n-1}) \leq d-m$, hence $\deg(f_n) = \deg(f_nx^n) = \deg(f-g_{n-1}h_{n-1}) \leq d$. 

By assumption, $\bar{g},\overline{h^{\phi^{n}}}$ are relatively prime in $K[t]$, hence we may choose $a,b \in R$ such that $\bar{a}\overline{g}+\bar{b}\overline{h^{\phi^{n}}} = 1$, hence $$\bar{a}\overline{g}\overline{f_n}+\bar{b}\overline{h^{\phi^{n}}}\overline{f_n} = \overline{f_n}.$$
Since $g$ is monic, we may apply left division with remainder in $R$ to find polynomials $q,p_n$ such that $bf_n = qg+p_n$ and $\deg(p_n) < m$. Let $q_n$ be the polynomial obtained from $af_n+qh^{\phi^{n}}$ by discarding all coefficients divisible by $x$. Then in the commutative ring $K[t]$, we have:

\begin{align*}
    \overline{gq_n + h^{\phi^{n}}p_n} 
    &= 
\bar{g}\overline{q_n} + \overline{h^{\phi^{n}}}\overline{p_n} = \bar{g}\bar{a}\overline{f_n}+\bar{g}\overline{q}\overline{h^{\phi^{n}}}+\overline{h^{\phi^{n}}}\overline{p_n}\\
    &= \overline{g}\bar{a}\overline{f_n} + \overline{h^{\phi^{n}}}\overline{qg+p_n}=\overline{g}\bar{a}\overline{f_n} + \overline{h^{\phi^{n}}}\bar{b}\overline{f_n}\\
    &= \overline{a}\bar{g}\overline{f_n} + \bar{b}\overline{h^{\phi^{n}}}\overline{f_n}=\overline{f_n}.
\end{align*}
Note that $\deg(\overline{f_n}) \leq(\deg(f_n)) \leq d$ and $$\deg(\overline{h^{\phi^{n}}p_n}) \leq \deg(h^{\phi^{n}}p_n) < (d-m)+m = d,$$ hence from the presented equality we get that $\deg(\overline{q_n}) \leq d-\deg(\bar{g}) = d-m$. By our construction of $q_n$ we have $\deg(q_n) = \deg(\overline{q}) \leq d-m$.

Now, we have the congruence: 
\begin{align*}
    f-g_nh_n &= f-(g_{n-1}+p_n x^n)(h_{n-1}+q_n x^n)\\
    &\equiv f-g_{n-1}h_{n-1}-(p_n x^n h_{n-1}+g_{n-1}q_n x^n)\\
    &= f_nx^n-(p_nh_{n-1}^{\phi^{n}}+g_{n-1}q_n)x^n = (f_n-(p_nh_{n-1}^{\phi^{n}}+g_{n-1}q_n))x^n
\end{align*}
modulo $x^{2n}$, hence also modulo $x^{n+1}$. Thus it remains to show that $\ord(f_n-(p_nh_{n-1}^{\phi^{n}}+g_{n-1}q_n)) \geq 1$, equivalently that $$\overline{f_n} = \overline{p_nh_{n-1}^{\phi^{n}}+g_{n-1}q_n} = \overline{p_n}\overline{h_{n-1}^{\phi^{n}}}+\overline{g_{n-1}}\overline{q_n} = \overline{h^{\phi^{n}}}\overline{p_n}+\overline{g}\overline{q_n} =\overline{gq_n + h^{\phi^{n}}p_n},$$
which we have already shown (in the presented equality we used the fact that $\bar{h} = \overline{h_{n-1}}$ implies that $\overline{h^{\phi^{n}}} = \overline{h_{n-1}^{\phi^{n}}}$, since $x^\phi = x$).
\end{proof}

We note that Hensel's lemma can be studied in a more general context, for rings that are complete with respect to a finitely generated ideal which need not be principal. We have chosen to limit ourselves here to the principal case in order to avoid additional technicalities -- the version provided here is the one needed to prove our main result in the next section.

\section{Proof of the skew Newton-Puiseux theorem}\label{sec:newton}

Let $F = \C((x^*)) = \bigcup_{n \geq 1} \C((x^{\frac{1}{n}}))$ be the field of Puiseux series. Let $\alpha$ be a positive real number, and for each $n\geq 1$ let $\sqrt[n]{\alpha}$ denote a positive $n$-th root of $\alpha$. Then we have a $\C$-automorphism $\sigma$ of $\C((x^*))$, given by $\sigma(x^{\frac{1}{n}}) = \sqrt[n]{\alpha}x^{\frac{1}{n}}$ for all $n \geq 1$, and so we may form the skew polynomial ring $F[t,\sigma]$. 

For any rational number $q = \frac{m}{n} \in \mathbb{Q}$, with $m \in \Z,n \in \N$, let $\sigma^q$ denote the $\C$-automorphism of $F$ determined by $x^\frac{1}{k} \mapsto (\sqrt[nk]{\alpha})^m x$. Clearly, these automorphisms are all well-defined and compatible with each other (that is, $\sigma^{qr} = (\sigma^q)^r$ for all $q,r \in \Q$).


For any $a \in F$ we denote by  $\delta_a \colon F \to F$ the map given by $\delta_a(b) = a(b^\sigma-b)$ for all $b \in F$. Then $\delta_a$ is a $\sigma$-derivation, called an inner $\sigma$-derivation (see \cite[Exercise 2Y, page 44]{GW04} and note that $a(b^\sigma-b) = ab^\sigma - ba = (-a)b-b^\sigma(-a)$, since $F$ is commutative).  

With the above notation in mind, for any $a \in F$ we may form the skew polynomial ring $F[t,\sigma,\delta_a]$.

\begin{lemma}\label{shift} Let $a,b \in F$. Consider the map $\vp \colon F[t,\sigma,\delta_a] \to F[t,\sigma,\delta_{a-b}]$ given by $\vp(\sum g_i t^i) = \sum g_i (t-b)^i$. Then $\vp$ is an isomorphism. \end{lemma} \begin{proof}
    Denote $T = F[t,\sigma,\delta_{a-b}]$. Then in the ring $T$ we have, for all $r \in F$:
    \begin{align*}
        (t-b)r &= r^\sigma t+\delta_{a-b}(r)-br = r^\sigma +(a-b)(r^\sigma-r)-br\\
        & = r^\sigma t+\delta_a(r)-r^\sigma b = r^\sigma(t-b)+\delta_a(r).
    \end{align*}
   Now apply Proposition \ref{goodearl}, with $\phi$ there taken as the identity and $y$ taken to be $t-b$. Then the isomorphism $\psi$ given by that proposition must coincide with the map $\vp$ here, since they agree on $t$ and on $F$.
\end{proof}

\begin{lemma}\label{dif} Let $a,b \in F,d \in \mathbb{N}$. The coefficient of $t^{d-1}$ in the polynomial $(t-b)^d \in F[t,\sigma,\delta_a]$ is $-\left(b+b^{\sigma}+b^{\sigma^{2}}+\ldots+b^{\sigma^{(d-1)}}\right)$. 
\end{lemma}

\begin{proof} For $d = 1$ there is nothing to prove. Suppose by induction that we have proven the claim for a given $d$. Clearly $(t-b)^d$ is monic, so we may write:

$$(t-b)^d = t^d - \left(b+b^{\sigma}+b^{\sigma^{2}}+\ldots+b^{\sigma^{(d-1)}}\right)t^{d-1} + \text{ lower degree terms},$$
hence

\begin{align*}
    (t-b)^{d+1} &= (t-b)\cdot \left(t^d - \left(b+b^{\sigma}+b^{\sigma^{2}}+\ldots+b^{\sigma^{(d-1)}}\right)t^{d-1} + \text{ lower degree terms}\right)\\
    &= t^{d+1}-t\left(b+b^{\sigma}+b^{\sigma^{2}}+\ldots+b^{\sigma^{(d-1)}}\right)t^{d-1} - bt^d + \text{ lower degree terms}\\
    &= t^{d+1}-\left(b+b^{\sigma}+b^{\sigma^{2}}+\ldots+b^{\sigma^{(d-1)}}\right)^{\sigma}\cdot t \cdot t^{d-1} - bt^d + \text{ lower degree terms}\\
    &= t^{d+1}-\left(b^{\sigma}+b^{\sigma^2}+b^{\sigma^{3}}+\ldots+b^{\sigma^{(d)}}\right)\cdot t^{d} - bt^d + \text{ lower degree terms}\\
    &= t^{d+1}-\left(b+b^{\sigma}+b^{\sigma^2}+b^{\sigma^{3}}+\ldots+b^{\sigma^{(d)}}\right)\cdot t^{d} + \text{ lower degree terms},
\end{align*}
%
%
%
%
%
%
as needed. 
\end{proof}

\begin{lemma}\label{surjective} Let $d \in \mathbb{N}$. The map $b  \mapsto b+b^{\sigma}+b^{\sigma^{2}}+\ldots+b^{\sigma^{(d-1)}}$ on $F$ is surjective. \end{lemma} \begin{proof} The image of $f = \sum f_j x^\frac{j}{m}$ under the given map is:

\begin{align*}
    \left(\sum f_j x^\frac{j}{m}\right)+\left(\sum f_j \alpha^{ \frac{j}{m}} x^\frac{j}{m}\right) &+\ldots + \left(\sum f_j \alpha^{\frac{((d-1))j}{m}} x^\frac{j}{m}\right)\\
    &= \sum f_j \left(1+\alpha^{\frac{j}{m}}+\ldots+ \alpha^{\frac{(d-1)j}{m}}\right)x^\frac{j}{m}\\
    &= \sum f_j \left(1+\beta_j+\ldots+ \beta_j^{d-1}\right)x^\frac{j}{m},
\end{align*}
%
%
%
%
where $\beta_j = \alpha^{\frac{j}{m}}$ for each index $j$. From this it is clear that the given map is surjective: The preimage of an arbitrary element $\sum g_j x^\frac{j}{m} \in F$ (for some $m \geq 1$) is given by $\sum g_j (1+\beta_j+\ldots+ \beta_j^{d-1})^{-1}x^\frac{j}{m}$. (Note that $1+\beta_j+\ldots+ \beta_j^{d-1}$ is always non-zero since $\alpha$ is a positive real number.)\end{proof}

\begin{lemma}\label{zero_coefficient} Let $a \in F$ and let $g = t^d+g_1t^{d-1}+\ldots+g_d \in F[t,\sigma,\delta_a]$. Let $b \in F$ with $b+b^\sigma+b^{\sigma^{2}}+\ldots+b^{\sigma^{(d-1)}} = g_1$, and let $\vp \colon F[t,\sigma,\delta_a] \to F[t, \sigma,\delta_{a-b}]$ be the isomorphism given by Lemma \ref{shift}. Then $\vp(g) = t^d+h_1t^{d-1}+\ldots+h_d$ satisfies $h_1 = 0$. \end{lemma} \begin{proof} We have $$\vp(g) = (t-b)^d+g_1(t-b)^{d-1}+\text{ terms of lower degree}$$
$$ = (t-b)^d+g_1t^{d-1}+\text{ terms of lower degree}, $$
and by Lemma \ref{dif} we have:

$$(t-b)^{d} = t^{d}-\left(b+b^{\sigma}+b^{\sigma^{2}}+\ldots+b^{\sigma^{(d-1)}}\right)t^{d-1}+\text {terms of lower degree},$$
hence $$\vp(g) = t^d+\left(g_1-\left(b+b^{\sigma}+b^{\sigma^{2}}+\ldots+b^{\sigma^{(d-1)}}\right)\right)t^{d-1}+\text{ terms of lower degree}.$$
\end{proof}

\begin{lemma}\label{lem:isom} Let $a \in F$ and let $r \in \Q$ be a rational number. Consider the map $\psi \colon F[t,\sigma,\delta_a] \to F[t,\sigma,\delta_{ax^r}]$ given by $\psi(\sum g_i t^i) = \sum g_i(x^{-r} t)^i$. Then $\psi$ is an isomorphism. \end{lemma} \begin{proof}
    Denote $T = F[t,\sigma,\delta_{ax^r}]$. In the ring $T$ we have, for all $f \in F$:
    $$(x^{-r}t)f = x^{-r}(f^\sigma t +\delta_{ax^r}(f)) = x^{-r}f^\sigma t+x^{-r}ax^r(f^\sigma-f)$$
    $$ = f^\sigma x^{-r} t+\delta_a(f).$$
   Now apply Proposition \ref{goodearl} (with $A = F$, $\phi$ being the identity automorphism on $F$, and $y$ taken to be $x^{-r}t$) to obtain the desired isomorphism, which must coincide with the map $\psi$ here since they agree on $t$ and on $F$.
\end{proof}

For any $0 \neq f \in \C((x^*))$ we let $\ord(f) \in \Q$ be the usual order of $f$ with respect to $x$, and put $\ord(0) = \infty$.

Given $a \in F$ and $n \in \mathbb{N}$, let $\C[[x^{\frac{1}{n}}]][t,\sigma,\delta_a]$ denote the subset of $F[t,\sigma,\delta_a]$ consisting of all polynomials in $t$ whose coefficients all belong to $\C[[x^{\frac{1}{n}}]]$. 

\begin{lemma}\label{subring} Let $n \in \mathbb{N}$. Suppose that $a \in \C[[x^{\frac{1}{n}}]]$. Then $\C[[x^{\frac{1}{n}}]][t,\sigma,\delta_a]$ is a subring of $F[t,\sigma,\delta_a]$. \end{lemma} \begin{proof} 
Note that the automorphism $\sigma$ restricts to an automorphism of $\C[[x^\frac{1}{n}]]$, and that $\delta_a(f) \in \C[[x^\frac{1}{n}]]$ for all $f \in \C[[x^\frac{1}{n}]]$, from which the claim is clear. \end{proof}

We note that the subring $\C[[x^{\frac{1}{n}}]][t,\sigma,\delta_a]$ considered in Lemma \ref{subring} coincides with the ring $A[t,\sigma,\delta]$ considered in \S\ref{skew_hensel} (with $A = \C[[x^{\frac{1}{n}}]], \delta = \delta_a, \frm = x^{\frac{1}{n}}A$) and that the assumptions there are satisfied here, namely $\sigma(\frm) \subseteq \frm, \delta(\frm) \subseteq \frm$ and $\bar{\sigma} = \id, \bar{\delta} = 0$. Indeed, let $f = \sum f_i x^\frac{i}{n} \in \C[[x^{\frac{1}{n}}]]$. Clearly, if $f \in \frm$ (that is, if $f_0 = 0$) then $\sigma(f) \in \frm$ and $\delta(f) = a(f^\sigma-f) \in \frm$, and we have $\bar{\sigma}(\bar{f}) = \bar{\sigma}(f_0) = f_0 = \bar{f}$ and $$\bar{\delta}(\bar{f}) = \overline{a(f^\sigma-f)} = \bar{a}\cdot \overline{(f^\sigma-f)}=\overline{a\cdot 0} = 0.$$

For any $a \in F$, we extend the $x$-adic order from $\C[[x]]$ to $\C[[x]][t,\sigma,\delta_a]$, by $\ord(\sum p_i t^i) = \min(\ord(p_i))$, as in \S\ref{skew_hensel}.

\begin{proposition}\label{changing_n} Let $n \in \mathbb{N}$, $\tau = \sigma^{\frac{1}{n}}$ and let $\xi \colon \C[[x^{\frac{1}{n}}]] \to \C[[x]]$ be the isomorphism given by $\sum a_ix^{\frac{i}{n}} \mapsto \sum a_ix^i$. Let $a \in \C[[x^{\frac{1}{n}}]]$ and let $b =  \xi(a) \in \C[[x]]$. Then $\xi$ extends to an isomorphism $\xi \colon \C[[x^{\frac{1}{n}}]][t,\sigma,\delta_a] \to \C[[x]][t,\tau,\delta_b]$ given by $\sum p_i t^i \mapsto \sum \xi(p_i)t^i$. \end{proposition} \begin{proof} For any $f \in \C[[x^{\frac{1}{n}}]]$ one checks directly that $\xi(f^\sigma) = \xi(f)^\tau$. Then in the ring $T = \C[[x]][t,\tau,\delta_b]$ we have $$t\xi(f) = \xi(f)^\tau t+ \delta_b(\xi(f)) = \xi(f^\sigma) t+ \xi(a)(\xi(f)^\tau-\xi(f))$$ $$= \xi(f^\sigma) t+ \xi(a\cdot(f^\tau-f)) = \xi(f^\sigma) t+ \xi(\delta_a(f)).$$
Now apply Proposition \ref{goodearl}, with $A, \phi, \delta,y$ there corresponding to $\C[[x^\frac{1}{n}]],\xi,\delta_a,t$ here, respectively. The obtained isomorphism must coincide with the map $\xi$ here, since they agree on $\C[[x^\frac{1}{n}]]$ and on $t$.
\end{proof}

Let $f \mapsto \bar{f}$ be the reduction map from $A=\C[[x]]$ to $\C = A/xA$, mapping a power series $f$ to its constant term $f_0$. By the definition of $\sigma$ it follows that $\overline{f^\sigma} = \overline{f}$ and that $\overline{\delta_a(f)} = 0$ for all $f \in \C[[x]]$. Then, as noted in the beginning of \S\ref{skew_hensel}, the map $p \mapsto \bar{p}$, defined by $\sum f_i t^i \mapsto \sum \overline{f_i}t^i$, is an epimorphism from $\C[[x]][t,\sigma,\delta_a]$ onto the (commutative) polynomial ring $\C[t]$.

We will need the following technical lemma:

\begin{lemma}\label{zero_sum} 
Let $d \in \N$ and let $c_1,\ldots,c_d$ be complex numbers belonging to the set $$\Gamma = \left\{\frac{d}{\alpha^{-n_1}+\alpha^{-n_2}+\ldots+\alpha^{-n_d}} - 1 \,\st\, n_1,\ldots,n_d \in \N \cup \{0\} \right\}.$$ If $c_1+c_2+\ldots+c_d = 0$, then $c_1 = c_2 = \ldots c_d = 0$. 
\end{lemma} 

\begin{proof} If $\alpha = 1$ then $\Gamma = \{0\}$ and the claim holds trivially. If  $\alpha > 1$ then the elements of $\Gamma$ are clearly all non-negative real numbers, and similarly if $\alpha < 1$ then the elements of $\Gamma$ are all non-positive real numbers. In both cases, from $\sum c_i = 0$ it follows that $c_1 = \ldots = c_d =0$.
\end{proof}

\begin{proposition}\label{factor} Let $a \in \C[[x]]$ and let $g = t^d+g_{d-1}t^{d-1}+\ldots+g_0$ be a monic polynomial of degree $d > 1$ in $\C[[x]][t,\sigma,\delta_a]$, such that $\ord(g_{d-1}) > 0$ and $\min_i\ord(g_i) = 0$. Then $g$ has a non-trivial factor in $\C[[x]][t,\sigma,\delta_a]$. \end{proposition} \begin{proof} Consider a factorization $(t-c_1)\cdot \ldots \cdot (t-c_d)$ of $\bar{g} \in \C[t]$ over $\C$. Since $\ord(g_{d-1}) > 0$ we have $\overline{g_{d-1}} = 0$ and hence $c_1+\ldots+c_d = 0$. Consider the set:

$$\Delta = \left\{a_0\left(\frac{d}{\alpha^{-n_1}+\alpha^{-n_2}+\ldots+\alpha^{-n_d}} - 1\right) \,\st\, n_1,\ldots,n_d \in \N \cup \{0\} \right\}.$$
We claim that at least one of $c_1,\ldots,c_d$ does not belong to $\Delta$. Indeed, suppose that $c_1,\ldots,c_d \in \Delta$. If $a_0 = 0$ then $\Delta = \{0\}$ and then $c_1 = \ldots = c_d = 0$. If $a_0 \neq 0$ then $\sum \frac{c_i}{a_0} = \frac{\sum c_i}{a_0} = 0$, hence by Lemma \ref{zero_sum}, applied for $\Gamma = \frac{1}{a_0}\Delta$, we get that $\frac{c_i}{a_0} = 0$, hence $c_i = 0$, for all $i$. Thus in both cases we have $\bar{g} = t^d$, which cannot be since at least one of the coefficients of $g$ has a non-zero constant term, by our assumptions

Thus at least one of the $c_i$ does not belong to $\Delta$, and we shall assume without loss of generality that $c_1 \notin \Delta$.

Let $T \colon \C \to \C$ be the map given by $T(w) = \alpha^{-1} w + a_0(\alpha^{-1}-1)$, where $a_0$ is the constant term of $a$. Then for all $n \in \N$ and $w \in \C$, we have $T^n(w) = \alpha^{-n} w + a_0(\alpha^{-n}-1)$. 

Let $\calC$ denote the orbit $\{c_1,T(c_1),T^2(c_1),\ldots\}$ of $c_1$ via $T$. We reorder $c_2,\ldots,c_d$ such that $c_1,\ldots,c_j \in \calC$ and $c_{j+1},\ldots,c_d \notin \calC$, for some $1 \leq j \leq d$, and write $c_i = T^{n_i}(c_1)$ for suitable $n_2,\ldots,n_j \in \N \cup \{0\}$, and $n_1 = 0$. We claim that $j < d$, so that $c_{j+1},\ldots,c_d \notin \calC$. Indeed, if $j = d$,  Then:

$$0 = c_1+\ldots+c_d = c_1+T^{n_2}(c_1)+\ldots+T^{n_d}(c_1)$$ $$ = c_1+(\alpha^{-n_2}c_1+a_0(\alpha^{-n_1}-1))+\ldots+(\alpha^{-n_d}c_1+a_0(\alpha^{-n_d}-1))$$

$$= c_1(1+\alpha^{-n_2}+\ldots+\alpha^{-n_d})+a_0((\alpha^{-n_2}+\ldots+\alpha^{-n_d})-(d-1)) $$ $$= (a_0+c_1)(1+\alpha^{-n_2}+\ldots+\alpha^{-n_d})-a_0d,$$
hence $(a_0+c_1)(1+\alpha^{-n_2}+\ldots+\alpha^{-n_d}) = a_0d$. If $a_0 = 0$ it follows that $c_1 =0$, but also $\Delta = \{0\}$, in contradiction to $c_1 \notin \Delta$. Thus $a_0 \neq 0$, and then we get

$$\big(1+\frac{c_1}{a_0}\big)(1+\alpha^{-n_2}+\ldots+\alpha^{-n_d}) = d,$$
therefore

$$c_1 = a_0\left(\frac{d}{1+\alpha^{-n_2}+\ldots+\alpha^{-n_d}}-1\right),$$
hence $c_1 \in \Delta$ (with $n_1 = 0$), a contradiction. 

Thus we must have $j < d$. Now consider the monic polynomials $\bar{u} = \prod_{i=1}^j(t-c_i), \bar{v} = \prod_{i=j+1}^d(t-c_i) \in \C[t]$. Let us write $\bar{u} = \sum u_i t^i, \bar{v} = \sum v_i t^i$, and let us trivially lift $\bar{u},\bar{v}$ to polynomials $u = \sum u_i t^i, v = \sum v_i t^i$ in $\C[[x]][t,\sigma,\delta_a]$.

Let $\phi$ denote the inner automorphism of the ring $\C((x))[t,\sigma,\delta_a]$ given by $p^\phi = xpx^{-1}$. By Lemma \ref{two-sided}, this automorphism restricts to an automorphism of $\C[[x]][t,\sigma,\delta_a]$. We claim that for all $n \in \N$, the polynomials $\bar{u},\overline{v^{\phi^n}}$ are coprime in $\C[t]$. Indeed, suppose to the contrary that for some $n \in \N$, the polynomial $\overline{v^{\phi^n}}$ vanishes at one of the zeros $c_1,\ldots,c_j$ of $\bar{u}$, say at $c_m = T^{n_m}(c_1)$. 

In the ring $\C[[x]][t,\sigma,\delta_a]$ we have $tx = \alpha x t+a(\alpha-1)x$, 
hence

$$xt = \alpha^{-1} tx-a(1-\alpha^{-1})x = (\alpha^{-1} t+a(\alpha^{-1}-1))x,$$ 
hence $xtx^{-1} = \alpha^{-1} t-a(1-\alpha^{-1})$. 
Thus $$\phi(t) = xtx^{-1} \equiv \alpha^{-1} t +a_0(\alpha^{-1}-1) \quad (\text{mod} \, x).$$ 
Since $\phi(x) = x$, we get by induction that $\phi^n(t) \equiv \alpha^{-n}t +a_0(a^{-n}-1) \quad (\text{mod} \, x)$. Then 
$$\phi^n(v) = \phi^n(\sum v_i t^i) = \sum v_i (\phi^n(t))^i \equiv \sum v_i (\alpha^{-n}t+a_0(a^{-n}-1))^i \quad (\text{mod} \, x),$$ 
hence $\overline{\phi^n(v)}$ is obtained from $\bar{v}$ by applying the isomorphism $t \mapsto \alpha^{-n}t +a_0(\alpha^{-n}-1)$ of $\C[t]$. Now, since $\overline{\phi^n(v)}$ vanishes at $c_m$, we get that $\bar{v}$ vanishes at $\alpha^{-n}c_m +a_0 (\alpha^{-n}-1) = T^n(c_m) = T^{n+n_m}(c_1)$, but this means that one of $c_{j+1},\ldots,c_d$ is $T^{n+n_m}(c_1)$, hence belongs to the orbit $\calC$, contradicting our choice of $j$.

Thus $\bar{u},\overline{v^{\phi^n}}$ are coprime in $\C[t]$. Finally, by Proposition \ref{Hensel} (with $A = \C[[x]], \delta = \delta_{a}$) , we get that $g$ has a non-trivial factor in $\C[[x]][t,\sigma,\delta_a] \subseteq F[t,\sigma,\delta_a]$. \end{proof}

\begin{theorem}\label{main} The field $F$ is algebraically closed with respect to $\sigma$. \end{theorem} \begin{proof} Let $f = \sum_{i=0}^d f_i t^i$ be a polynomial in $F[t,\sigma]$, which we assume without loss of generality to be monic, of degree $d \geq 2$. We shall show that $f$ has a non-trivial factor, from which the theorem follows by induction on the degree. If $f_i = 0$ for all $i$, then $f = t^d = t\cdot t^{d-1}$ factors, and so we may assume that at least one of the $f_i$ is non-zero. Thus $r = \max\{\frac{-\ord(f_i)}{d-i} \st i \in \{0,1,\ldots,d-1\}\}$ is a rational number. By Lemma \ref{lem:isom} (applied with $a = 0$) we have an isomorphism $\psi \colon F[t,\sigma] \to F[t,\sigma]$ which maps $f$ to $\psi(f) = (x^{-r} t)^d + f_{d-1}(x^{-r} t)^{d-1}+\ldots+f_0$. For all $i \geq 0$ we have $(x^{-r}t)^i =  \beta_i x^{-ri}t^i$, where $\beta_i = \alpha^{-\frac{i(i+1)}{2}}$ (proof by induction on $i$), hence 

$$\psi(f) = \beta_d x^{-rd} t^d + \beta_{d-1}f_{d-1}x^{-r(d-1)} t^{d-1}+\ldots+f_0,$$ hence 

$$\beta_d^{-1}x^{rd}\psi(f) = t^d + \beta_d^{-1}\beta_{d-1}f_{d-1}x^{r} t^{d-1}+\ldots+\beta_d^{-1}f_0x^{rd},$$

and for all $0 \leq i \leq d-1$ we have, by our choice of $r$, $\ord(\beta_d^{-1}\beta_{i}f_ix^{(d-i)r}) = \ord(f_i)+r(d-i) \geq 0$, and moreover equality holds for at least one index $i$. 

Thus, by replacing $f$ with $\beta_d^{-1}x^{rd}\psi(f)$, we may assume that $\ord(f_i) \geq 0$ for all $0 \leq i \leq d-1$ and that $\ord(f_i) = 0$ for some $i$. If $\ord(f_{d-1}) > 0$ then $f$ has a non-trivial factor by Proposition \ref{factor}. Thus we may assume that $\ord(f_{d-1}) = 0$. By Lemma \ref{surjective}, there exists $b \in F$ such that $b+b^\sigma+b^{\sigma^{2}}+\ldots+b^{\sigma^{(d-1)}} = f_{d-1}$, and clearly we have $\ord(b) = \ord(f_{d-1}) = 0$. Consider the isomorphism $\vp \colon F[t,\sigma] \to F[t,\sigma,\delta_{-b}]$ given by Lemma \ref{shift}. The coefficients of $\vp(f) = \sum f_i(t-b)^i \in F[t,\sigma,\delta_{-b}]$ are all of non-negative order, since the order of $b$ and of all of the $f_i$ is non-negative. Let us write $\vp(f) = \sum f_i' t^i$. Since $\vp$ preserves degrees in $t$, it suffices to find a non-trivial factor for $\vp(f)$. As mentioned, $\ord(f_i') \geq 0$ for all $i$, and we have $f_{d-1}' = 0$ by Lemma \ref{zero_coefficient}. In particular, $\ord(f_{d-1}') = \infty > 0$. 

Now, the element $b$ and all of the $f_i'$ belong to $\C[[x^\frac{1}{n}]]$, for some sufficiently large $n\in \N$. By applying the isomorphism given in Proposition \ref{changing_n}, we may assume that $n = 1$. 

If there exists $0 \leq i <d-1$ such that $\ord(f_i') = 0$, then $\vp(f)$ has a non-trivial factor in $\C[[x]][t,\sigma,\delta_{-b}]$, again by Proposition \ref{factor}. Thus we may assume that $\ord(f_i') > 0$ for all $i$. Let $\phi$ denote the inner automorphism of $\C((x))[t,\sigma,\delta_{-b}]$ given by $p^\phi = xpx^{-1}$. By Lemma \ref{two-sided} (with $A = \C[[x]], \delta = \delta_{-b}$), this automorphism restricts to an automorphism of $\C[[x]][t,\sigma,\delta_{-b}]$

Now consider the epimorphism $p \mapsto \bar{p}$ from $\C[[x]][t,\sigma,\delta_{-b}]$ onto $\C[t]$. Let $g = t^{d-1},h = t$. We have $\overline{\vp(f)} = t^d = t^{d-1} \cdot t = \bar{g}\bar{h}$. We claim that for all $n \in \N$, the polynomials $\bar{g},\overline{h^{\phi^n}}$ are coprime in $\C[t]$. Indeed, in the ring $\C[[x]][t,\sigma,\delta_{-b}]$ we have $$t x = \alpha x t + (-b)(\alpha^{-1}-1)x,$$ hence $$\phi(t) = xtx^{-1} = \alpha^{-1}t+b(1-\alpha^{-1}),$$

hence $\overline{\phi(t)} = \alpha^{-1}t+b_0(1-\alpha^{-1})$, where $b_0$ is the constant term of $b$. It follows by induction that 

$$\overline{\phi(t)^n} = \alpha^{-n}t+b_0(1-\alpha^{-1})(1+\alpha^{-1}+\ldots+\alpha^{-(n-1)}),$$ 

for all $n \in \N$. We may assume that $\alpha \neq 1$ (if $\alpha = 1$ we are in the well-known commutative case) and we have $b_0 \neq 0$, since $\ord(b) = 0$. Thus the presented polynomial is coprime to $\bar{g} = t^{d-1}$ in $\C[t]$ for all $n \in \N$. Finally, by Proposition \ref{Hensel} the polynomial $\vp(f)$ has a non-trivial factor, as needed. \end{proof}

\begin{remark}\label{non_real} 
In Theorem \ref{main} one may replace $\C$ by any algebraically closed field containing $\C$, and the proof remains the same. 
However, we may not replace the real number $\alpha$ which determines the automorphism $\sigma$ by an arbitrary non-zero complex number: 
For any $\alpha \in \C^\times$, and for each $k\geq 1$ let $\sqrt[k]{\alpha}$ denote a $k$-th root of $\alpha$ such that $\sqrt[kn]{\alpha}^k = \sqrt[n]{\alpha}$ for all $n,k \geq 1$. (If $\alpha = re^{\rm{i} \theta}$ is given in polar\footnote{We use the symbol $\rm{i}$ to denote the complex number $\sqrt{-1}$, as opposed to the letter $i$ which we reserve for indices.} 
form, take $\sqrt[k]{\alpha} = \sqrt[k]{r}e^{\rm{i}\theta/k}$, where $\sqrt[k]{r} \in \R$ is a positive $k$-th root of $r$).  
Then we have the $\C$-automorphism $\sigma$ of $\C((x^*))$, given by $\sigma(x^{\frac{1}{k}}) = \sqrt[k]{\alpha}x^{\frac{1}{k}}$ for all $k \geq 1$. 
One may ask whether the field $F$ is algebraically closed with respect to all such automorphisms. 
The answer is generally negative: For example, for $\alpha = \rm{i}$, the polynomial $t^2-(1+x^2) \in F[t,\sigma]$ does not have a $\sigma$-zero. 
Indeed, suppose to the contrary that $f \in F$ satisfies $f^\sigma f = 1+x^2$. Then clearly $f$ has order $0$, and we may write $f = f_0 + f_q x^q+\ldots$, with $f_0 = \pm 1$ and $f_q$ being the second non-zero coefficient that appears in $f$, for some $0 < q \in \Q$. Then $$1+x^2 = f^\sigma f = 1+f_0f_q(1+{\rm{i}}^q)x^{q}+\ldots$$
and clearly we must have $q \leq 2$. If $q < 2$, then $f_0f_q(1+\rm{i}^q) \neq 0$, contradicting the presented equality, and for $q = 2$ we get $1+\rm{i}^q = 0 \neq 1$, also a contradiction.  Thus $t^2-(1+x^2)$ does not have a $\sigma$-zero. \end{remark}

The example in Remark \ref{non_real} raises the following theoretical question: Does there exist a field $F$ which is algebraically closed with respect to all automorphisms of $F$? We suspect that such a field does not exist.

We conclude this work with the following remark:
\begin{remark}\label{final}
Consider the polynomial $p = (t-1)^2 = t^2-2t+1 \in F[t,\sigma]$. The only $\sigma$-zero of $f$ is $1$. Indeed, for $\alpha = 1$ (that is, $\sigma = \id$) the claim is obvious. Suppose that $\alpha \neq 1$ and that $f = \sum_{i \geq m} f_i x^\frac{i}{n}$ with $f_m \neq 0$ is a $\sigma$-zero of $p$. Then 

$$f^\sigma f - 2f+1 = (\sum_{i \geq m} f_i \alpha^\frac{i}{n}x^\frac{i}{n})(\sum_{i \geq m} f_i x^\frac{i}{n}) -2(\sum_{i \geq m} f_i x^\frac{i}{n})+1 = 0.$$

By comparing coefficients we immediately see that $m = 0$ and $f_0 = 1$. The coefficient of $x^\frac{1}{n}$ in the left-hand side of the presented equality is then $f_1+\alpha^\frac{1}{n} f_1 -2f_1 = (\alpha^\frac{1}{n}-1)f_1$, hence we must have $f_1 = 0$. Suppose by induction that $f_1 = \ldots = f_k = 0$ for some $k \geq 1$. Then the coefficient of $x^\frac{k+1}{n}$ in the left-hand side of the presented equality is $f_{k+1}\alpha^\frac{k+1}{n} + f_{k+1}-2f_{k+1} = f_{k+1}(\alpha^\frac{k+1}{n} - 1)$, hence $f_{k+1} = 0$. Thus $f_i = 0$ for all $i \geq 1$. 

This example (together with Theorem \ref{main}, for $\alpha \neq 1$) resolves the question of Aryapoor mentioned in Remark \ref{open} in \S\ref{prelim}.
\end{remark}


\begin{thebibliography}{13}
\expandafter\ifx\csname natexlab\endcsname\relax\def\natexlab#1{#1}\fi
\providecommand{\url}[1]{\texttt{#1}}
\providecommand{\href}[2]{#2}
\providecommand{\path}[1]{#1}
\providecommand{\DOIprefix}{doi:}
\providecommand{\ArXivprefix}{arXiv:}
\providecommand{\URLprefix}{URL: }
\providecommand{\Pubmedprefix}{pmid:}
\providecommand{\doi}[1]{\href{http://dx.doi.org/#1}{\path{#1}}}
\providecommand{\Pubmed}[1]{\href{pmid:#1}{\path{#1}}}
\providecommand{\bibinfo}[2]{#2}
\ifx\xfnm\relax \def\xfnm[#1]{\unskip,\space#1}\fi
\bibitem[{Abhyankar(1990)}]{Ab90}
\bibinfo{author}{Abhyankar, S.S.}, \bibinfo{year}{1990}.
\newblock \bibinfo{title}{Algebraic geometry for scientists and engineers},
  \bibinfo{publisher}{American Mathematical Society}.
  volume~\bibinfo{volume}{35} of \textit{\bibinfo{series}{Mathematical Surveys
  and Monographs}}.
\bibitem[{Aryapoor(2009)}]{Ar09}
\bibinfo{author}{Aryapoor, M.}, \bibinfo{year}{2009}.
\newblock \bibinfo{title}{Non-commutative henselian rings}.
\newblock \bibinfo{journal}{Journal of Algebra} \bibinfo{volume}{322},
  \bibinfo{pages}{2191–2198}.
\bibitem[{Aryapoor(2023)}]{Ar23}
\bibinfo{author}{Aryapoor, M.}, \bibinfo{year}{2023}.
\newblock \bibinfo{title}{On some notions of algebraically closed
  $\sigma$-fields}.
\newblock \bibinfo{journal}{Journal of Algebra and its applications} .
\bibitem[{Bergen et~al.(2015)Bergen, Giesbrecht, Shivakumar and Zhang}]{BGS15}
\bibinfo{author}{Bergen, J.}, \bibinfo{author}{Giesbrecht, M.},
  \bibinfo{author}{Shivakumar, P.}, \bibinfo{author}{Zhang, Y.},
  \bibinfo{year}{2015}.
\newblock \bibinfo{title}{Factorizations for difference operators}.
\newblock \bibinfo{journal}{Advances in difference equations}
  \bibinfo{volume}{57}, \bibinfo{pages}{1--6}.
\bibitem[{Cohn(1991)}]{Coh91}
\bibinfo{author}{Cohn, P.}, \bibinfo{year}{1991}.
\newblock \bibinfo{title}{Algebraic numbers and algebraic functions}.
\newblock \bibinfo{publisher}{Taylor and Francis}.
\bibitem[{Cohn(1995)}]{Coh95}
\bibinfo{author}{Cohn, P.}, \bibinfo{year}{1995}.
\newblock \bibinfo{title}{Skew Fields: Theory of General Division Rings}.
\newblock \bibinfo{publisher}{Cambridge University Press}.
\bibitem[{Goodearl and Warfield(2004)}]{GW04}
\bibinfo{author}{Goodearl, K.R.}, \bibinfo{author}{Warfield, R.B.},
  \bibinfo{year}{2004}.
\newblock \bibinfo{title}{An Introduction to Noncommutative Noetherian rings}.
\newblock \bibinfo{publisher}{Cambridge University Press}.
\newblock \bibinfo{note}{2nd ed.}
\bibitem[{Lam and Leroy(1988)}]{LL88}
\bibinfo{author}{Lam, T.}, \bibinfo{author}{Leroy, A.}, \bibinfo{year}{1988}.
\newblock \bibinfo{title}{Vandermonde and wronsksian matrices over division
  rings}.
\newblock \bibinfo{journal}{Journal of Algebra} \bibinfo{volume}{119},
  \bibinfo{pages}{308--336}.
\bibitem[{Ore(1933a)}]{Ore33b}
\bibinfo{author}{Ore, O.}, \bibinfo{year}{1933}a.
\newblock \bibinfo{title}{On a special class of polynomials}.
\newblock \bibinfo{journal}{Transactions of the American Mathematical Society}
  \bibinfo{volume}{35}, \bibinfo{pages}{559--584}.
\bibitem[{Ore(1933b)}]{Ore33}
\bibinfo{author}{Ore, O.}, \bibinfo{year}{1933}b.
\newblock \bibinfo{title}{Theory of non-commutative polynomials}.
\newblock \bibinfo{journal}{Annals of Mathematics} \bibinfo{volume}{34},
  \bibinfo{pages}{480--508}.
\bibitem[{Pumpl{\"u}n(2015)}]{Pum15}
\bibinfo{author}{Pumpl{\"u}n, S.}, \bibinfo{year}{2015}.
\newblock \bibinfo{title}{Factoring skew polynomials over hamilton's quaternion
  algebra and the complex numbers}.
\newblock \bibinfo{journal}{Journal of Algebra} \bibinfo{volume}{427},
  \bibinfo{pages}{20--29}.
\bibitem[{Smith(1977)}]{Sm77}
\bibinfo{author}{Smith, K.}, \bibinfo{year}{1977}.
\newblock \bibinfo{title}{Algebraically closed noncommutative polynomial
  rings}.
\newblock \bibinfo{journal}{Communications in algebra} \bibinfo{volume}{5},
  \bibinfo{pages}{331--346}.
\bibitem[{Veluscek(2010)}]{Vel10}
\bibinfo{author}{Veluscek, D.}, \bibinfo{year}{2010}.
\newblock \bibinfo{title}{Higher product of pythogoras numbers of skew fields}.
\newblock \bibinfo{journal}{Asian-European Journal of mathematics}
  \bibinfo{volume}{3}, \bibinfo{pages}{193–207}.

\end{thebibliography}


\end{document}